\documentclass{article}
\usepackage{authblk}
\usepackage{amsmath,amsfonts,amssymb,amscd,amsthm,amsbsy,upref}
\usepackage{xcolor}   

\newtheorem{thm}{Theorem}[section]
\newtheorem{lem}[thm]{Lemma}
\newtheorem{cor}[thm]{Corollary}
\newtheorem{prop}[thm]{Proposition}

\newtheorem{prob}{Problem}
\theoremstyle{definition}
\newtheorem{defin}[thm]{Definition}

\newtheorem*{thmA}{Theorem A}
\newtheorem*{thmB}{Theorem B}

\newcommand{\N}{\mathbb{N}}

\newcommand{\R}{\mathbb{R}}

\newcommand{\cC}{\mathcal C}

\newcommand{\cF}{\mathcal F}

\newcommand{\cU}{\mathcal U}
\newcommand{\cV}{\mathcal V}

\newcommand{\keq}{\!=\!}

\newcommand\kin{\!\in\!}

\newcommand{\Lip}{\text{{\rm Lip}}}

\newcommand{\vp}{\varepsilon}

\newcommand{\ie}{\textit{i.e.,}\ }

\allowdisplaybreaks
\begin{document}

\title{Coarse embeddings into $c_0(\Gamma)$}
\author[1]{Petr H\'ajek\thanks{The first named author was supported by GA\v CR 16-073785 and
 RVO: 67985840.}} 
\author[2]{Th.~Schlumprecht\thanks{ The second named  author  was supported by the National Science Foundation under the Grant Number  DMS--1464713.}}

\affil[1]{Institute of Mathematics, Academy of Science of the Czech Republic,  \v Zitn\'a 25 115 67 Prague~1,
Czech Republic, 
and Department of Mathematics, Faculty of Electrical Engineering,
Czech Technical University in Prague,  Zikova 4, 160 00 Prague, Czech Republic.

{\it email: \tt hajek@math.cas.cz}}

\affil[2]{Department of Mathematics, Texas A\&M University, College Station, TX 77843, USA
and Department of Mathematics, Faculty of Electrical Engineering,
Czech Technical University in Prague, Zikova 4, 160 00 Prague, Czech Republic.

{\it email: \tt schlump@math.tamu.edu}}

\maketitle
\begin{abstract}  Let $\lambda$ be  a  large enough cardinal number
 (assuming GCH it suffices to let $\lambda=\aleph_\omega$).
 If $X$ is a Banach space with
$\text{dens}(X)\ge\lambda$, which admits a coarse (or uniform)
embedding into any $c_0(\Gamma)$, then $X$ fails to have nontrivial
cotype, i.e. $X$ contains  $\ell_\infty^n$ $C$-uniformly for every $C>1$.
In the special case  when $X$
has a symmetric basis, we may even conclude that it is linearly
 isomorphic with
$c_0(\text{dens}X)$.
\end{abstract}

\section{Introduction}\label{S:1}

The classical result of Aharoni states that every separable
metric space (in particular every separable Banach space) can be bi-Lipschitz
embedded (the definition is given below) into $c_0$.

The natural problem of embeddings of metric spaces into $c_0(\Gamma)$,
for an arbitrary set $\Gamma$, has been
treated by several authors, in particular Pelant and Swift. The
characterizations that they obtained,
and which  play a crucial role in our argument, are described below.

Our main interest, motivated by
some problems posed in \cite{GLZ},
lies in the case of embeddings of Banach spaces into $c_0(\Gamma)$.

We now state the main results of this paper. We first define the following
cardinal numbers inductively.
We put $\lambda_0=\omega_0$, and, assuming for $n\in\N_0$, $\lambda_n$ has
been defined, we put $\lambda_{n+1}=2^{\lambda_n}$.
Then we let

\begin{equation}\label{def-lam}
\lambda=\lim_{n\to\infty} \lambda_n.
\end{equation}
It is clear that
assuming the generalized continuum hypothesis (GCH) $\lambda=\aleph_\omega$.

\begin{thmA}
If $X$ is a Banach space with density
$\text{dens}(X)\ge\lambda$, which admits a coarse (or uniform)
embedding into any $c_0(\Gamma)$, then $X$ fails to have nontrivial
cotype, i.e. $X$ contains  $\ell_\infty^n$ $C$-uniformly for some
 $C>1$ (equivalently, every $C>1$).
\end{thmA}

Our method of proof gives a much stronger result for Banach
spaces with a symmetric basis. Namely, under the assumptions
of Theorem A, such spaces are linearly isomorphic with $c_0(\Gamma)$
(Theorem \eqref{sym-case}).

Theorem A will follow from the following combinatorial result
which  is of independent interest.

\begin{thmB} Assume that $\Lambda$ is a set whose cardinality is at least
$\lambda$, $n\in\N$, and
$\sigma:[\Lambda]^n\to \cC$ is a map into an arbitrary set $\cC$.
Then (at least) one of the following conditions holds:
\begin{enumerate}
\item There is a  sequence $(F_j)_{j=1}^\infty$ of pairwise
disjoint elements of $[\Lambda]^n$, so that
$\sigma(F_i)=\sigma(F_j)$, for all $i,j\in\N$.
\item There is an $F\in[\Lambda]^{n-1}$ so that
$\sigma\big(\big\{ F\cup\{\gamma\}: \gamma\in\Lambda\big\}\big)$ is infinite.
\end{enumerate}
\end{thmB}

The above Theorem B was previously deduced in \cite{PHK} from a combinatorial result
of Baumgartner, provided $\Lambda$ is a weakly compact cardinal number
(whose existence is not provable in ZFC, as it is inaccessible \cite{Je} p. 325, p. 52).
The authors in \cite{PHK} pose a question whether assuming that $\Gamma$ is
uncountable is sufficient in Theorem B.

Theorem B is used in order to obtain a scattered compact set $K$ 
of height $\omega_0$, such that $C(K)$
 does not uniformly embed into $c_0(\Gamma)$. It is easy to check that
our version of Theorem B implies a ZFC example of such a $C(K)$ space.
It is further shown in \cite{PHK} that the space $C[0,\omega_1]$
does not uniformly embed into any $c_0(\Gamma)$.

Let us point out that a special case of Theorem A
was obtained by Pelant and Rodl \cite{PR}, namely it was shown there that
$\ell_p(\lambda), 1\le p<\infty$,
spaces (which are well known to have nontrivial cotype)
do not uniformly embed into any $c_0(\Gamma)$.

The paper is organized as follows. In Section \ref{S:2} we recall
Pelant's \cite{Pe,PHK} amd Swift's \cite{Sw} conditions   for
Lipschitz, uniform, and  coarse embeddability into $c_0(\Gamma)$.
 In Section \ref{S:3} we provide a proof for Theorem B.
 Finally, in Section \ref{S:4}, we provide a proof
 of Theorem A as well as the symmetric version of the result.

All set theoretic concepts and results used in our note can be found in
\cite{Je}, whereas for facts concerning nonseparable Banach spaces
\cite{HMVZ} can be consulted.

\section[Pelant's and Swift's criteria]{Pelant's and Swift's criteria  for Lipschitz, uniform, and  coarse embeddability into $c_0(\Gamma)$}\label{S:2}

In this section we recall some of the notions  and results  by 
Pelant \cite{Pe,PHK} and Swift \cite{Sw} about embeddings into $c_0(\Gamma)$.

For a metric space $(M,d)$ a {\em  cover} is a set $\cU$ of subsets of $M$ 
such that $M=\bigcup_{U\in\cU} U$. 
A  cover $\cU$  of $M$ is called {\em uniform}  if
there is an $r>0$ so that for all $x\in M$ there is a $U\in\cU$, 
so that  $B_r(x)=\{x'\in M: d(x',x)<r\}\subset U$. 
It is   called {\em uniformly  bounded} if the diameters of the $U\in\cU$
are uniformly bounded, and it is called {\em point finite} if every $x\in M$ 
lies in only finitely many $U\in\cU$. 
A cover $\cV$ of $M$ is a {\em refinement } of a cover $\cU$, 
if for every $V\in\cV$
there is a $U\in\cU$, for which $V\subset U$.

\begin{defin}{\rm\cite{PHK}}  A metric space $(M,d)$ is said to have the {\em Uniform Stone Property} if every uniform cover  $\cU$ of $M$ has a point finite uniform  refinement.
\end{defin}

\begin{defin}{\rm\cite{Sw}}  A metric space $(M,d)$ is said to have the {\em Coarse Stone Property} if every bounded cover is the refinement of a point finite uniformly bounded cover.
\end{defin}

\begin{defin} Let $(M_1,d_1)$ and $(M_2,d_2)$ be two metric spaces. For a map $f:M_1\to M_2$ we define the  {\em modulus of uniform continutiy} $w_f:[0,\infty)\to [0,\infty]$, and the
{\em modulus of expansion} $\rho:[0,\infty)\to [0,\infty]$ as follows
\begin{align*}
w_f(t)&=\sup\big\{ d_2\big(f(x),f(y)\big): x,y\in M_1,\, d_1(x,y)\le t\big\} \text{ and } \\
\rho_f(t)&=\inf\big\{ d_2\big(f(x),f(y)\big): x,y\in M_1,\, d_1(x,y)\ge t\big\}.
\end{align*}
The map  $f$ is called {\em uniform continuous } if  
$\lim_{t\to 0} w_f(t)=0$, and it is called a {\em uniform embedding} 
if moreover $\rho_f(t)>0$ for every $t>0$. It is called {\em coarse} if
$w_f(t)<\infty$, for all $0<t<\infty$ and it is called a {\em coarse embedding}, 
if $\lim_{t\to\infty}\rho_f(t)=\infty$.
The map $f$ is called {\em Lischitz continuous} if
$$\Lip(f)=\sup_{x\not=y} \frac{d_2(f(x),f(y))}{d_1(x,y)} <\infty,$$
and a bi-Lipschitz embedding, if $f$ is injective and $\Lip(f^{-1})$ is also finite.
\end{defin}
The following result recalls results from \cite{PHK}(for (i)$\iff$(ii)) and \cite{Sw} (for (ii)$\iff$(iii)$\iff$(iv)$\iff$(v)).
\begin{thm} For a Banach space $X$ the following properties are equivalent.
\begin{enumerate}
\item[(i)] $X$ has the uniform Stone Property.
\item[(ii)] $X$ is uniformly embeddable into $c_0(\Gamma)$, for some 
set $\Gamma$.
\item[(iii)] $X$ has the coarse Stone Property.
\item[(iv)] $X$ is coarsely embeddable  into $c_0(\Gamma)$, for some
set $\Gamma$.
\item[(v)] $X$ is bi-Lipschitzly embeddable into  $c_0(\Gamma)$, for some
set $\Gamma$.
\end{enumerate}
\end{thm}
It is easy to see, and  was noted in \cite{PHK,Sw}, that the uniform
Stone property and the coarse Stone property are inherited by subspaces.
The equivalence (i)$\iff$(ii) was used in \cite{PHK} to show that
$C[0,\omega_1]$ does not uniformly embed in any $c_0(\Gamma)$. It was also used
  to prove that certain other $C(K)$-spaces do not  uniformly embed into 
  $c_0(\Gamma)$:
Let $\Lambda$ be any set and denote  for $n\in\N$ by $[\Lambda]^{\le n}$ and $[\Lambda]^n$ the subsets of $\Lambda$ which have cardinality  at most $n$ and
exactly $n$, respectively. Endow $[\Lambda]^{\le n}$ with the restriction of the product topology on $\{0,1\}^\Lambda$ (by identifying each set with its characteristic  function).
Then define $K_\Lambda$ to be the  one-point  Alexandroff compactification of  the topological sum  of the spaces $[\Lambda]^{\le n}$, $n\in\N$.
It was shown in \cite{PHK} that if $\Lambda$ satisfies Theorem B
 then
$C(K_\Lambda)$ is not uniformly Stone and thus does not embed  uniformly into 
any $c_0(\Gamma)$.

\section{A combinatorial argument}\label{S:3}

We start by introducing property $P(\alpha)$ for a cardinal 
$\alpha$ as follows.
\begin{align}
&\text{For every $n\in\N$ and any  map $\sigma:[\alpha]^n\to \cC$, $\cC$ being an arbitrary set,} \tag{$P(\alpha)$}\\
&\text{(at least) one of the following two conditions hold:} \notag \\
&\text{There is a sequence $(F_j)$  of pairwise  disjoint elements of $[\lambda]^n$, with $\sigma(F_i)\keq\sigma(F_j)$,}  \label{E:3.1}\\
&\text{for  any $i,j\in\N$.}\notag\\
&\text{There is an $F\in [\lambda]^{n-1}$, so that $\sigma\big(\big\{ F\cup\{\gamma\}: \gamma\in \lambda\setminus F\big\}\big)$ is infinite.} \label{E:3.2}
\end{align}

As remarked in Section \ref{S:2}, if $\kappa$ is an uncountable  weakly
compact cardinal number, then $P(\kappa )$ holds. But the existence of
weakly compact
cardinal numbers requires further set theoretic axioms, beyond ZFC \cite{Je}.
In
 \cite[Question 3]{PHK} the authors ask if $P(\omega_1)$ is true.
 \begin{thm}\label{T:3.1}For  $\lambda$ defined by \eqref{def-lam},
 $P(\lambda)$ holds.   \end{thm}

 For our proof of Theorem \ref{T:3.1} it will be more convenient to reformulate it into a statement about $n$-tuples, instead of sets of cardinality $n$.
 We will first introduce some notation.

 Let $n\in\N$ and  $\Gamma_1$, $\Gamma_2,\ldots \Gamma_n$ be sets  of infinite cardinality, and put $\Gamma=\prod_{i=1}^n \Gamma_i$.  For $a\in  \Gamma$ and $1\le i\le n$
 we denote the $i$-the coordinate of $a$ by $a(i)$. We say that two points $a$ and $b$ in $\Gamma$ are {\em diagonal}, if $a(i)\not=b(i)$, for all $i\in\{1,2,\ldots,n\}$.

  Let $a\in \Gamma_i$ for $i\in\N$.
  For $i\in\{1,2\ldots,n\}$ we call the set
  $$H(a,i)=\big\{(b_1,b_2,\ldots, b_{i-1}, a(i),b_{i+1}, \ldots ,\alpha_{n}): b_j\in \Gamma_j,\text{ for $j\in\{1,2\ldots n\}\setminus\{i\}$}\big\},$$
  the {\em Hyperplane through the point $a$ orthogonal to $i$}.
  We call the set
  $$L(a,i)=\big\{(a(1),\ldots,a(i-1),b_i,a(i+1), \ldots a(n)): b_i\in \Gamma_i\big\},$$
 the {\em Line through the point $a$ in direction of $i$}.

 For   a cardinal number $\beta$, we define  recursively the following  sequence of cardinal numbers
 $\big(\exp_+(\beta,n):n\kin\N_0$\big): $\exp_+(\beta,0)=\beta$, and, assuming $\exp_+(\beta,n)$ has been defined for some $n\in\N_0$, we put
 $$\exp_+(\beta,n+1)= \big(2^{\exp_+(\beta,n)^+}\big)^+.$$
 Here $\gamma^+$ denotes the {\em successor cardinal}, for a cardinal $\gamma$, \ie the smallest cardinal number  $\gamma'$ with $\gamma'>\gamma$.
 Note that since $\exp_+(\gamma,1)\le 2^{2^{2^\gamma}}$, it follows  for the above defined cardinal number $\lambda$, that
 $$\lambda =\lim_{n\to\infty} \exp_+(\omega_0,n).$$
 Secondly, successor cardinals are regular \cite{Je}, and thus every set of cardinality  $\gamma$, with  $\gamma$ being a successor cardinal, can
 be partitioned  for $n\in\N$ into $n$ disjoint sets 
 $\Gamma_1$, $\Gamma_2,\ldots, \Gamma_n$, all of them 
 having also cardinality $\gamma $, and the
 map  
 $\Gamma_1\times \Gamma_2\times\ldots\times \Gamma_n\to 
 \big[ \bigcup_{i=1}^n  \Gamma_i\big]^n$,  $(a_1,a_2,\ldots, a_n)
 \mapsto \{a_1,a_2,\ldots a_n\}$, is  injective.
  We therefore deduce  that the following statement   implies Theorem \ref{T:3.1}.

  \begin{thm}\label{T:3.2} Let $n\in\N$ and a  assume that the sets 
  $\Gamma_1$, $\Gamma_2,\ldots \Gamma_n$ have cardinality at 
  least $\exp_+(\omega_1, n^2)$. For any function
  $$\sigma: \Gamma:=\prod_{i=1}^n\Gamma_i\to \cC,$$
 where $\cC$ is an arbitrary set,
at least one of the following two  conditions hold

  \begin{align}
  \label{E:3.2.1} &\text{There is a sequence $(a^{(j)})_{j=1}^\infty$, of pairwise diagonal elements in $\Gamma$,
  so that}\\
  &\text{$\sigma(a^{(i)})=\sigma(a^{(j)})$, for any $i,j\in\N$.}\notag\\
  \label{E:3.2.2} &\text{There is a line $L\subset \Gamma$, 
  for which $\sigma(L)$ is infinite.}
  \end{align}
  \end{thm}
  We will make  the following observation before proving Theorem \ref{T:3.2}.
  \begin{lem}\label{L:3.3}
   Let $n\in\N$ and  $\Gamma_1$, $\Gamma_2,\ldots \Gamma_n$  be non empty sets.
Let
   $$\sigma: \Gamma:=\prod_{i=1}^n\Gamma_i\to \cC,$$
be a function that fails both conditions \eqref{E:3.2.1} and \eqref{E:3.2.2}.

   Then there is a set  $\tilde\cC$ and a function
    $$\tilde\sigma: \Gamma:=\prod_{i=1}^n\Gamma_i\to\tilde \cC,$$
that fails both \eqref{E:3.2.1} and \eqref{E:3.2.2} and moreover has 
the property that
    \begin{align}\label{E:3.3.1} \text{for every $c\in\cC$ there is a hyperplane $H_c\subset \Gamma$ so that $\{b\in \Gamma: \tilde\sigma(b)=c\}\subset H_c$.}\end{align}
  \end{lem}
  \begin{proof} We may assume without loss of generality that $\sigma$ is 
  surjective. Since \eqref{E:3.2.1} is not satisfied  for each $c\in\cC$
there exists an $m(c)\in\N$
  and a (finite) sequence  
  $(a^{(c,j)})_{j=1}^{m(c)}\subset \sigma^{-1}(\{c\})$, 
  which is pairwise diagonal, and  maximal, with this property. 
Hence
  $$\sigma^{-1}(\{c\})\subset\bigcup_{j=1}^{m(c)} \bigcup_{i=1}^n H(a^{(c,j)},i).$$
  Indeed, from  the maximality of  $(a^{(c,j)})_{j=1}^{m(c)}\subset \sigma^{-1}(\{c\})$, it follows that each $b\in\sigma^{-1}(\{b\})$ must have at least one coordinate in common with at least
  one element of  $(a^{(c,j)})_{j=1}^{m(c)}\subset \sigma^{-1}(\{c\})$.

  We define
  $$\tilde\cC=\bigcup_{c\in \cC} \{1,2,\ldots, m(c)\}\times\{1,2\ldots n\}\times \{c\},$$
  and
  \begin{align*}
  &\tilde\sigma:\Gamma\to \tilde\cC,\quad b\mapsto (c,j,i),\text{ where }\\
  &c=\sigma(b),\quad j=\min \Big\{ j': b\in  \bigcup_{i'=1}^n H(a^{(c,j')},i')\Big\}, \text{ and } i= \min\{ i': b\in H(a^{(c,j)},i)\}.
  \end{align*}
  It is clear that $\tilde\sigma$ satisfies \eqref{E:3.3.1}.  Since for every $c\in\cC$,
  $$\{b\in \Gamma: \sigma(b)=c\}=\bigcup_{j=1}^{m(c)}\bigcup_{i=1}^n \{ b\in \Gamma: \tilde\sigma(b)=(c,j,i)\},$$
  $\tilde\sigma$ does not satisfy \eqref{E:3.2.1}.
   In order to verify  that \eqref{E:3.2.2}  is not satisfied, assume 
   $L\subset \Gamma$ is a line, and let $\{c_1,c_2,\ldots,c_p\}$ be 
   the image
   of $L$ under $\sigma$. By construction,
   $$\tilde\sigma(L)\subset \{ (j,i,c_k),  k\le p, \, j\le m(c_k), \i\le n\},$$
   which is also finite.
\end{proof}
\begin{proof}[Proof of Theorem \ref{T:3.2}] We assume  that 
$\sigma: \Gamma=\Gamma_1\times\Gamma_2\times\ldots \times\Gamma_n\to \cC$ 
is a map which fails both
\eqref{E:3.2.1} and \eqref{E:3.2.2}. By Lemma \ref{L:3.3} we may also assume 
that $\sigma$ satisfies \eqref{E:3.3.1}. For each $a\in \Gamma$ we 
 fix
an $i(a)\in \{1,2,\ldots, n\}$ so that $\sigma^{-1}\big(\{\sigma(a)\} \big)\subset H(a,i(a))$. It is important to note that, since \eqref{E:3.2.2} is not satisfied, it follows that  each
line $L$, whose direction is some $j\in\{1,2,\ldots,n\} $, can 
only have finitely many elements $b$ for which $i(b)=j$. Indeed, if $i(b)=j$
 then $b$ is uniquely determined by the value  $\sigma(b) $.
To continue  with  the the proof the following {\em Reduction Lemma} will be essential.

\begin{lem}\label{L:3.4} Let $\beta$ be an uncountable  regular  
cardinal. Assume that  $\tilde\Gamma_1\subset \Gamma_1$, 
$\tilde\Gamma_2\subset \Gamma_2,\ldots,\tilde\Gamma_n\subset \Gamma_n$
 are such that
 $|\tilde\Gamma_i|\ge \exp_+(\beta,n)$, for all $i\in\{1,\dots,n\}$. 

Then, for any $i\in\{1,2 \ldots,n \}$
 there are   a  number $K_i\in\N$, and subsets 
 $\Gamma_1'\subset \tilde\Gamma_1, \Gamma_2'\subset\tilde\Gamma_2,\ldots, 
 \Gamma_n'\subset\tilde\Gamma_n,$
 with $|\Gamma'_j|\ge \beta$, so that
 \begin{align}\label{E:3.4.1}
 \forall (a_1,a_2,\ldots, a_{i-1},a_{i+1},&\ldots a_n)\kin \prod_{j=1, j\not=i}^n\Gamma_i'\quad\\
  &\big|\{a\in \Gamma_i': i(a_1,a_2,\ldots a_{i-1},a,a_{i+1},\ldots a_n)=i\}\big|\le K_i.\notag
 \end{align}
\end{lem}
\begin{proof} We assume  without loss of generality that $i=n$. Abbreviate  $\beta_j=\exp_+(\beta,j)$, for $j=1,2\ldots n$. We
 choose subsets $\tilde\Gamma_j^{(0)}\subset \tilde \Gamma_j$, for which $\big|\tilde\Gamma^{(0)}_j\big|=\beta_{n+1-j}$.

Since the $\beta_j$'s are regular, it follows for each $j=1,2\ldots ,n-1$ that
 \begin{align*}
\big|\tilde\Gamma^{(0)}_j\big|&=\beta_{n+1-j}\\
                                   &>2^{\beta_{n-j}}\\
                                   &=2^{|\tilde\Gamma_{j+1}^{(0)}\times \tilde\Gamma_{j+1}^{(0)}\times\ldots\times \tilde\Gamma^{(0)}_n|}\\
                                   &=\big|\{ f: \tilde\Gamma_{j+1}^{(0)}\times \tilde\Gamma_{j+1}^{(0)}\times\ldots\times \tilde\Gamma^{(0)}_n\to \N\}\big|.
 \end{align*}
 Abbreviate for $i=1, \ldots , n$.
 $$\cF_{j}=\{ f: \tilde\Gamma_{j}^{(0)}\times \tilde\Gamma_{j+1}^{(0)}\times\ldots\times \tilde\Gamma^{(0)}_n\to \N\} $$
and consider the function
$$\phi_1: \prod_{j=1}^{n-1} \tilde\Gamma_j^{(0)}\to\N,\quad (a_1,a_2,\ldots, a_{n-1})\mapsto\big|\{ a\kin \tilde\Gamma_n^{(0)}: i(a_1,a_2,\ldots,a_{n-1},a)\keq n\}\big|.$$
 For fixed $a_1\in \Gamma_1^{(0)}$,
  $\phi_1(a_1,\cdot)\in \cF_2$, and the cardinality of $\cF_2$ is by 
  the above estimates smaller than the cardinality of 
  $\tilde\Gamma_1^{(0)}$, which is regular.
Therefore we can find a function $\Phi_2\in \cF_2$  and a subset $\Gamma'_1\subset \tilde\Gamma_i^{(0)}$ of cardinality $\beta_n$ so that
 $\phi_1(a_1,\dot)=\phi_2$ for all $a_1\in\Gamma'_1$. We continue the process and find $\Gamma'_{j}\subset \tilde\Gamma_{j}^{(0)}$, for $j=1,2\ldots, n-2$ of cardinality
 $\beta_{n+1-j}$ and functions $\phi_j\in F_j$, for $j=1,2,\ldots ,n-1$, so that for all $(a_1,a_2,\ldots, a_{n-2})\in \prod_{j=1}^{n-2}\Gamma'_j$ and $a_{n-1}\in \tilde\Gamma_{n-1}^{(0)}$, we have
 \begin{align}
 \phi_1(a_1,a_2,\ldots, a_{n-1})=\phi_2(a_2,\ldots, a_{n-1})=\ldots= \phi_{n-1}(a_{n-1}) .
 \end{align}
 Then, since $\phi_{n-1}$ is $\N$ valued, we can finally  choose an $K_n\in\N$  and a subset $\Gamma'_{n-1}$, of cardinality   at least $\beta$,     so that $\phi_{n-1}(a_1)\le K_{n}$, which finishes our argument.
\end{proof}

\medskip\noindent{\em Continuation of the proof of Theorem \ref{T:3.2}}.  
We apply Lemma \ref{L:3.4} successively to all $i\in\{1,2\ldots n\}$, 
and the cardinals
$\beta^{(i)}= \exp_+(\omega_1, n(n-i))$. 
We obtain numbers $K_1,K_2,\ldots K_n$ in $\N$ and infinite sets
$\Lambda_j\subset \Gamma_i$, for $j\le n$, so that for all $i=1,2\ldots  n$  and all $a\in\prod_{j=1}^n \Lambda_j$
$$\big|\{ a\in\Lambda_i: i(a_1,a_2,\ldots, a_{i-1},a, a_{i+1},\ldots a_n)=i\}\big|\le K_i.$$
In order to deduce a contradiction choose for each $j=1,\ldots n$  a subset $A_j$ of $\Lambda_j$ of cardinality $l_j=(n+1)K_j$.
Then it follows that
\begin{align*}
\prod_{j=1} l_j&= \Big|\prod_{j=1}^n A_j\Big|\\
&=\sum_{i=1}^n \sum_{ a\in\prod_{j=1,j\not=i}^n A_j} \big|\{a\in A_i: i(a_1,a_2,\ldots, a_{i-1},a, a_{i+1},\ldots a_n)=i\}\big|\\
&\le \sum_{i=1}^n K_i\prod_{j=1, j\not=i}^n l_j\le\frac{n}{n+1}\prod_{j=1}^nl_j
\end{align*}
which is a contradiction and finishes the proof of the Theorem.
\end{proof}

We can now state the ZFC version of Theorem 4.1. in \cite{PHK}, in which it was shown that for weakly compact cardinalities 
   $\kappa_0$ the space $C(K_{\kappa_0})$,  where $K_{\kappa_0}$ was defined at the end of Section \ref{S:2},
 cannot be uniformly
(or coarsely) embedded into any $c_0(\Gamma)$,  where $\Gamma$ has  any cardinality.
 Since the only property of $\kappa_0$, which is needed in \cite{PHK}, is the fact  that $(P(\kappa_0))$ holds, we deduce 
 \begin{cor}\label{C:2.8}
 $C(K_\lambda) $ does not coarsely  (or uniformly) embed into $c_0(\Gamma)$, for any cardinality $\Gamma$.
\end{cor}

\section{Proof of Theorem A}\label{S:4}

In this section we use our combinatorial Theorem B from Section \ref{S:3} to show Theorem A.

Recall that a long Schauder basis of a Banach space $X$ is a transfinite
sequence $\{e_\gamma\}_{\gamma=0}^\Gamma$ such that for every $x\in X$ there exists
a unique transfinite sequence of scalars $\{a_\gamma\}_{\gamma=0}^\Gamma$
such that $x=\sum_{\gamma=0}^\Gamma a_\gamma e_\gamma$.
Similarly,  a long Schauder basic sequence in a Banach space $X$ is a transfinite
sequence $\{e_\gamma\}_{\gamma=0}^\Gamma$ which is a long Schauder basis
of its closed linear span.
Recall that the $w^*-\text{dens}(X^*)$
is the smallest cardinal such that there exists a $w^*$-dense subset of $X^*$.
Analogously to the classical Mazur construction of a Schauder basic sequence
in a separable Banach space we have the following result, proved e.g.
in \cite[p.135]{HMVZ}  (the fact that the basis is normalized, i.e.
$\|e_\gamma\|=1$,
is not a part of the statement in \cite{HMVZ}, but it is
easy to get it by normalizing the existing basis).
\begin{thm}
Let $X$ be a Banach space with $\Gamma=w^*-\text{dens} X^*>\omega_0$.
Then $X$ contains a monotone normalized
long Schauder basic sequence of length $\Gamma$.
\end{thm}

\begin{proof}[Proof of Theorem A]
Using the Hahn-Banach theorem it is easy to see
that $w^*-\text{dens} X^*\le \text{dens} X$.
On the other hand, since every $x\in X$ is uniquely determined by its
values on a $w^*$-dense subset of $X^*$, it is clear that

\[
\text{dens} X\le\text{card} X\le 2^{w^*-\text{dens} X^*}
\]

It follows that for $\lambda$ defined in \eqref{def-lam} we get
that $\lambda=w^*-\text{dens} X^*$ holds if and only if
 $\lambda=\text{dens} X$.
In order to prove Theorem A
we may assume without loss of generality that
$X$ has a long  normalized  and monotone Schauder basis  $(e_\mu)_{\mu<\lambda}$, of length $\lambda$, i.e. $\Gamma=\lambda$. 

Set
\[
D_n=\{F\subset\lambda: |F|=n\},\;\;n\in\N
\]

Suppose that $F=\{\gamma_1,\dots,\gamma_n\}$ where $\gamma_1<\dots<\gamma_n$
are elements of $[0,\lambda)$ arranged in an increasing order.
Consider the corresponding
finite set $M_F=\{\sum_{i=1}^n\vp_i e_{\gamma_i}: \vp_i\in\{-1,1\}\}$,
containing $2^n$ distinct vectors of $X$,
and put a linear order $\prec$ on this set according to the arrangement
of the signs $\vp_i$,  setting
\[
\sum_{i=1}^n\vp_i e_{\gamma_i}\prec
\sum_{i=1}^n\tilde{\vp}_i e_{\gamma_i}
\]
if and only if for the minimal $i$, such that $\vp_i\ne\tilde{\vp}_i$,
it holds $\vp_i<\tilde{\vp}_i$.
In order to prove Theorem A it suffices to show that if $M=\cup_{F\in D_n, n\in\N} M_F\subset X$
has the coarse Stone property then $X$ fails to have nontrivial cotype.
To this end,
starting with $\cU=\{ B_2(x): x\in M\}$ we find a uniform bounded cover
$\cV$, which is point finite and  so that $\cU$ refines $\cV$, \ie
for all $x\in M$ there is a $V_x\in \cV$ with $B_2(x)\subset V_x$.
Let $r>0$ be such that each $V\in \cV$ is  a subset of a ball of radius $r$.

Let $\mathcal C$ be the set  consisting of all finite tuples $(V^1,\dots, V^m)$,
where $V^j\in\mathcal V$.
We now define the function $\sigma:M\to\mathcal C$ as follows.
If $F\in D_n, F=\{\gamma_1,\dots,\gamma_n\}$ where $\gamma_1<\dots<\gamma_n$,
we let

\begin{equation}\label{sig-def}
\sigma(F)=(V_{y_1},\dots,V_{y_{2^n}}),
\end{equation}

where $y_1\prec\dots\prec y_{2^n}$ are the elements of $M_F$
arranged in the increasing order.
Applying Theorem B to the function $\sigma$, for a fixed $n\in\N$,
yields one of two possibilities.
Either
there is an $F=\{\gamma_1,\dots,\gamma_{n-1}\}$, where
$\gamma_1<\dots<\gamma_{n-1}$, so that
$\sigma\big(\big\{ F\cup\{\tau\}: \tau\in \lambda\setminus F\big\}\big)$
is infinite. In this case, pick an infinite sequence of distinct $\{\tau_j\}_{j=1}^\infty$
witnessing the desired property.
By passing to a subsequence, we may assume without loss of generality that
either there exists $k, 1\le k\le n-1$, so that for all $j\in\N$,
$\gamma_k<\tau_j<\gamma_{k+1}$, or $\tau_j<\gamma_1$ for all $j\in\N$,
or $\gamma_{n-1}<\tau_j$ for all $j\in\N$.
For simplicity of notation, assume the last case,
i.e. $\gamma_1<\dots<\gamma_{n-1}<\tau_j$ holds for all $j\in\N$.
Denoting $F^j=\{\gamma_1,\dots,\gamma_{n-1},\tau_j\}$, we
conclude that there exists a fixed selection of signs
$\vp_1,\dots,\vp_n$ such that the set
\[
B=\Big\{V_y: y=\sum_{i=1}^{n-1}\vp_i e_{\gamma_i}+\vp_n\tau_j, j\in\N\Big\}
\]
is infinite. Indeed, otherwise the set of values
$\{\sigma(\{\gamma_1,\dots,\gamma_{n-1},\tau_j\}), j\in\N\}$,
which are determined by the definition \eqref{sig-def}, would have only a finite
set of options for each coordinate, and would therefore have to be finite.
This is a contradiction with the point finiteness of the system $\mathcal V$,
because
\[
\sum_{i=1}^{n-1}\vp_i e_{\gamma_i}\in V_y,\;\;\text{for all}\; V_y\in B.
\]

It remains to consider the other case when
there is a sequence $(F_j)$  of pairwise  disjoint elements
of $[\lambda]^n$, with $\sigma(F_i)\keq\sigma(F_j)$,
for  any $i,j\in\N$. In fact, it suffices to choose just a pair
of such disjoint elements (written in an increasing order of ordinals)
$F=\{\gamma_1,\dots,\gamma_{n}\}$, $G=\{\beta_1,\dots,\beta_{n}\}$,
such that $\sigma(F)=\sigma(G)$.
This means, in particular, that for every
fixed selection of signs
$\vp_1,\dots,\vp_n$,
\[
V_{\sum_{i=1}^{n}\vp_i e_{\gamma_i}}=
V_{\sum_{i=1}^{n}\vp_i e_{\beta_i}}
\]
By our assumption, the elements of $\mathcal V$ are contained in a ball
of radius $r$, hence
\begin{equation}\label{all-sign}
\|\sum_{i=1}^{n}\vp_i e_{\gamma_i}-\sum_{i=1}^{n}\vp_i e_{\beta_i}\|\le2r
\end{equation}
holds for any selection of signs
$\vp_1,\dots,\vp_n$.
Let $u_j=e_{\gamma_j}-e_{\beta_j}$, $j\in\{1,\dots,n\}$. Because
$\{e_\gamma\}$ is a monotonne normalized long Schauder basis, we have
the trivial estimate $1\le\|u_j\|\le2$.
The equation \eqref{all-sign} means that

\begin{equation}\label{all-sign-2}
1\le \|\sum_{i=1}^{n}\vp_i u_i\|\le2r
\end{equation}
holds for any selection of signs
$\vp_1,\dots,\vp_n$.  Since norm functions are convex, this means that for the unite vector ball $B_E$ of of $E=\text{span}(u_i:i\le )$ it  follows that 
$$\Big\{ \sum_{j=1}^n a_j u_j: |a_j|\le \frac1{2r}\Big\}\subset  B_E \subset\Big\{\sum_{j=1}^n a_j u_j :|a_j|\le 2\Big\},$$
which means that $(u_j)_{j=1}^n$ is $4r$-equivalent to the unit vector basis of $\ell_\infty^n$.
\end{proof}

In fact, our proof gives a much stronger condition than just failing cotype,
because our copies of $\ell_\infty^k$ are formed by vectors of the type
$e_\alpha-e_\beta$. This fact can be used to obtain
much stronger structural results for spaces with special bases.
Recall that a long Schauder basis $\{e_\gamma\}_{\gamma=1}^\Lambda$
is said to be symmetric if
\[
\|\sum_{i=1}^n a_i e_{\gamma_i}\|=
\|\sum_{i=1}^n a_i e_{\beta_i}\|
\]
for any selection of $a_i\in\R$, and any pair of sets
$\{\gamma_i\}_{i=1}^n\subset[1,\Lambda)$,
$\{\beta_i\}_{i=1}^n\subset[1,\Lambda)$. It is well-known (c.f. 
\cite[ Prop. II.22.2]{Si}), that each symmetric basis is
automatically unconditional, i.e. there exists $K>0$ such that
\[
\frac1K\|\sum_{i=1}^n |a_i| e_{\gamma_i}\|\le
\|\sum_{i=1}^n a_i e_{\gamma_i}\|\le K
\|\sum_{i=1}^n |a_i| e_{\gamma_i}\|.
\]
In particular, 
\[
\frac1K\|\sum_{i\in A} a_i e_{\gamma_i}\|\le
\|\sum_{i\in B} a_i e_{\gamma_i}\|
\]
whenever $A\subset B$.

\begin{thm}\label{sym-case}
Let $X$ be a Banach space of density $\Lambda\ge\lambda$, with a symmetric basis
$\{e_\gamma\}_{\gamma=1}^\Lambda$, which coarsely (or uniformly)
embeds into some $c_0(\Gamma)$. 

Then $X$ is linearly isomorphic
with $c_0(\Lambda)$.
\end{thm}
\begin{proof}

By the proof of the above results, if $X$ embeds into $c_0(\Gamma)$,
there exists an $C>0$, such that for each $k\in\N$ there are some vectors
$\{v_i\}_{i=1}^k$ of the form $v_i=e_{\gamma_i}-e_{\beta_i}$ satisfying
 the conditions

\begin{equation}\label{V-f}
\frac12\max_j|a_i|\le
\Big \|\sum_{i=1}^k a_i v_i\Big\|
 \le C\max_j|a_j|.
\end{equation}

Using the fact that the basis $\{e_\gamma\}$ is unconditional,
(and symmetric) we obtain by an easy manipulation that there exist some
constants $A,B>0$ such that

\begin{equation}\label{sssi}
A\Big\|\sum_{i=1}^k a_i e_{\gamma_i}\Big\|\le
\Big\|\sum_{i=1}^k a_i v_i\Big\|\le B\Big\|\sum_{i=1}^k a_i e_{\gamma_i}\Big\|
\end{equation}

Combining \eqref{V-f} and \eqref{sssi} we finally obtain
that for some  $D\ge 1$, and any $k\in\N$, 
\begin{equation}\label{V-f-2}
\frac1D\max_j|a_i|\le
\|\sum_{i=1}^k a_i e_{\gamma_i}\|\le
D\max_j|a_j|.
\end{equation}
for all  $\{\gamma_1,\dots,\gamma_k\}\subset[1,\Lambda)$,
which proves our claim.
\end{proof}

\section{Final comments and open problems}\label{S:5}

Let us mention in this final section some problems of interest.

First of all, we do not know whether or not Theorem A is true if we replace $\lambda$ by smaller cardinal numbers.

\begin{prob}\label{Prob:1} Assume that $X$ is  a Banach space with $\text{\rm dens}(X)\ge \omega_1$, and assume that $X$ coarsely embeds into $c_0(\Gamma)$ for some cardinal number  $\Gamma$.
Does $X$ have trivial co-type? If moreover $X$ has a symmetric basis, must it be isomorphic  to $c_0(\omega_1)$?
\end{prob} 

Of course  Problem \ref{Prob:1} would have a positive answer if the following is true.

\begin{prob}\label{Prob:2} Is Theorem B true for $\omega_1$?
\end{prob}

Connected to Problems  \ref{Prob:1} and \ref{Prob:2} is the following 
\begin{prob}\label{Prob:3} Does $\ell_\infty$ coarsely embed into $c_0(\kappa)$ for some uncountable cardinal number $\kappa$.
\end{prob}

Another   line of interesting problems  asks which isomorphic properties do  non separable  Banach spaces have which  coarsely embed into $c_0(\Gamma)$

\begin{prob}\label{Prob:4} Does a non separable Banach space which  coarsely embeds into some $c_0(\Gamma)$, $\Gamma$  being uncountable, contain copies $c_0$, or even $c_0(\omega_1)$.
\end{prob}

\end{document}